\documentclass[a4paper,12pt]{article}
\usepackage{amsfonts}
\usepackage{amssymb}
\usepackage{amsmath}

\usepackage{graphics}
\usepackage{tgbonum}
\usepackage{tikz}
\usepackage{flexisym}
\usepackage{cite}
\usepackage{nth}
\usepackage{amsthm}
\usepackage{hyperref}
\usetikzlibrary{decorations.pathmorphing}

\newtheorem{theorem}{Theorem}
\newtheorem{lemma}[theorem]{Lemma}

\newtheorem{mydef}{Definition}
\newtheorem{mydef*}{Definition*}

\usetikzlibrary {positioning}
\usepackage{pifont,geometry,txfonts,hyperref}
\usepackage[utf8]{inputenc}
\usepackage[english]{babel}

 \DeclareUnicodeCharacter{2212}{-}

\title{ON THE LINEAR INDEPENDENCE OF RADICALS}
\author{Sourav Koner$^1$\\harakrishnaranusourav@gmail.com$^1$ \and Dhiren Kumar Basnet$^2$\\dbasnet@tezu.ernet.in$^2$}
\date{Department of Mathematical Sciences\\Tezpur University, Tezpur, Assam, India, 784028}

\begin{document}
\maketitle

\begin{abstract}
We provide an alternative proof that the finite rational linear combination of radicals, under certain constraint, are linearly independent over $\mathbb{Q}$.   
\end{abstract}


\section{Introduction}
An irrational number is a real number that cannot be expressed as a fraction with the numerator as integers and denominator as nonzero integers. One of the most famous irrational number is $\sqrt{2}$, sometimes called Pythagoras's constant. Proof of the irrationality of $\sqrt{2}$ can be obtained in the following way: assume $\sqrt{2}$ is rational, that is, it can be expressed as a fraction of the form $\frac{w}{y}$, where $w$ and $y$ are two relatively prime positive integers. Now, since $\sqrt{2} = \frac{w}{y}$, we have $2 = \frac{w^{2}}{y^{2}}$, or $w^{2} = 2y^{2}$. Since $2y^{2}$ is even, $w^{2}$ must be even and since $w^{2}$ is even, so is $w$. Let $w = 2z$. We have $4z^{2} = 2y^{2}$ and thus $y^{2} = 2z^{2}$. Since $2z^{2}$ is even, $y^{2}$ is even, and since $y^{2}$ is even, so is y. However, two even numbers cannot be relatively prime, so $\sqrt{2}$ cannot be expressed as a rational fraction; hence $\sqrt{2}$ is irrational. Similarly, proving that the number $\sqrt{2} + \sqrt{3}$ is irrational can be done in the following manner: let $\sqrt{2} + \sqrt{3}$ be a rational number, say $x$. $x = \sqrt{2} + \sqrt{3}$ implies $x - \sqrt{2} = \sqrt{3}$. Squaring on both sides, we obtain, $x^{2} + 2 - 2x\sqrt{2} = 3$, or $\frac{x^{2} - 1}{2x} = \sqrt{2}$. Now, $\frac{x^{2} - 1}{2x}$ is a rational number. But this contradicts the fact that $\sqrt{2}$ is an irrational number. So, our supposition is false. Therefore, $\sqrt{2} + \sqrt{3}$ is an irrational number. The techniques which we have used to establish the irrationality of the above two numbers, if we use the same techniques on other rational linear combination of radicals, for example $\pm \sqrt[8]{12} \pm \frac{2}{3}\sqrt[27]{108} \pm \frac{1}{2}\sqrt{12}$, then it would become cumbersome. 

Let $U$ denotes the set of all radicals which are irrationals, that is, $$U = \{\sqrt[m]{b} \mid b \in \mathbb{Q}^{+}, m \in \mathbb{N} - \{1\}, \sqrt[m]{b} \notin \mathbb{Q}\}.$$ 
Further, if $\mathcal{S}$ denotes the set of all finite rational linear combination of radicals in $U$ such that if $\alpha \in \mathcal{S}$ then the terms in the expression of $\alpha$ do not trivially cancel out by simplifying the radicals in the expression of $\alpha$ (for example, we do not consider $3\sqrt{12} - 5\sqrt{3} - \sqrt[4]{9}$ to be an element of $\mathcal{S}$ as $3\sqrt{12} - 5\sqrt{3} - \sqrt[4]{9} = 6\sqrt{3} - 5\sqrt{3} - \sqrt{3} = 0$), then we prove the following.

\begin{theorem}
If $\alpha \in \mathcal{S}$ then $\alpha$ cannot be expressed as $\frac{p}{q}$, where $p \in \mathbb{Z}$ and $q \in \mathbb{N}$.
\end{theorem} 

An equivalent form of the above theorem has already been proved\cite{1940}\cite{1974}, which says that finite rational linear combination of radicals, under certain constraint, are linearly independent. The main motivation behind this note is to provide an elegant alternative proof of Theorem 1. 

\begin{mydef}
An element $\sqrt[m]{b} \in U$ is said to be a reduced irrational if it can not be written of the form $e\sqrt[n]{d}$ where $n \in \mathbb{N}$, $e, d \in \mathbb{Q}^{+}$ and $n < m$.
\end{mydef}

For example, $\sqrt[3]{\frac{81}{5}}, \sqrt[4]{9}$ are not a reduced irrational numbers as $\sqrt[3]{\frac{81}{5}} = 3\sqrt[3]{\frac{3}{5}}$ and $\sqrt[4]{9} = \sqrt{3}$, whereas $\sqrt[3]{\frac{9}{5}}, \sqrt[4]{3^{3}}$ is a reduced irrational number. 

\begin{lemma}
Let $\sqrt[m]{b}$ be a reduced irrational number and $l_{0}, l_{1}, \ldots, l_{t} \in \mathbb{Q}$ be such that $l_{t} \ne 0$. If $t < m$, then $l_{0} + l_{1}\sqrt[m]{b} + \cdots + l_{t}\sqrt[m]{b^{t}}$ is an irrational number.
\end{lemma} 

\begin{proof}  
Suppose $l_{0} + l_{1}\sqrt[m]{b} + \cdots + l_{t}\sqrt[m]{b^{t}} = \frac{p}{q}$ for some $p \in \mathbb{Z}$ and $q \in \mathbb{N}$. If we consider the polynomials $u(X), v(X) \in \mathbb{Q}[X]$ given by, $u(X) = ql_{t}X^{t} + \cdots + ql_{1}X + ql_{0} - p$ and $v(X) = X^m - b$, then observe that $\sqrt[m]{b}$ is a common zero for both of the polynomials. Therefore $gcd(u(X), v(X)) = r(X)$ exists in $\mathbb{Q}[X]$. Since $r(X)$ divides $u(X)$ and $v(X)$, so we have $deg(r) < m$ and the zeros of $r(X)$ are also the zeros of $v(X)$. If $ r_{0} \in \mathbb{Q} - \{0\}$  be the constant term of the polynomial $r(X)$, then $\prod{(\omega\sqrt[m]{b})} = r_{0}$, where the product is taken over all zeros of $r(X)$ and $\omega$ is some $m$-th root of unity. Now taking modulus on both sides, we obtain that $|r_{0}| = (\sqrt[m]{b})^{deg(r)}$, that is, $\sqrt[deg(r)]{|r_{0}|} = \sqrt[m]{b}$. But then this contradicts the fact that $\sqrt[m]{b}$ is a reduced irrational number. Therefore it must be that our supposition is false. This completes the proof.
\end{proof}

Observe that from the above lemma it is easy to understand that what can be the minimal polynomial for a reduced irrational number $\sqrt[m]{b}$. That is, if we take $p = 0$ and $q = 1$ in the above proof, then it tells us that the degree of the minimal polynomial cannot be less than $m$, and it is exactly $m$, namely, $X^{m} - b$. Further consider the finite product $c = \prod_{i}(\sqrt[m_{i}]{b_{i}})^{\epsilon_{i}}$ with $0 \le \epsilon_{i} \le (m_{i} - 1)$ such that $c \notin \mathbb{Q}$ and each $\sqrt[m_{i}]{b_{i}}$ is a reduced irrational number. Since $c^{\prod_{i}m_{i}} \in \mathbb{N}$, so by well ordering principle we can have a smallest positive integer $s$, such that $c^{s} \in \mathbb{Q}$. If we set $c^{s} = k$, then observe that $\sqrt[s]{k}$ is the reduced irrational number and that $X^{s} - k$ is the minimal polynomial for $c$ over $\mathbb{Q}$.

Now, let $\mathcal{I_{L}}$ denotes the set of all reduced irrational numbers. 
 
 \begin{mydef}
 For $k \in \mathbb{N}$, a subset $S = \{\sqrt[m_{i}]b_{i} \mid 1 \le i \le k\}$ of $\mathcal{I_{L}}$ is said to be a reduced set if and only if $\prod_{i = 1}^{k}(\sqrt[m_{i}]{b_{i}})^{\epsilon_{i}} \notin \mathbb{Q}$, $\forall$ $(\epsilon_{1}, \epsilon_{2}, \ldots, \epsilon_{k}) \in V_{1} \times V_{2} \times \cdots \times V_{k} - \{0\}$, where $V_{i} = \{0, 1, 2, \ldots, (m_{i} - 1)\}$ and $0 = (0, 0, \ldots, 0)$.
 \end{mydef}
 Observe that every nonempty subset of a reduced set is again a reduced set.  
 
\begin{lemma}
If $\alpha \in \mathcal{S}$, then there exist a reduced set $S$ such that $\alpha \in \mathbb{Q}(S)$.
\end{lemma}

\begin{proof}
Let $\alpha = r_{1}\sqrt[m_1]{c_{1}} + r_{2}\sqrt[m_2]{c_{2}} + \cdots + r_{n}\sqrt[m_n]{c_{n}}$. Assume, without loss of generality, that $\sqrt[m_{i}]{c_{i}} \ne \sqrt[m_{j}]{c_{j}}$ whenever $i \ne j$ where $1 \le i, j \le n$. let $c_{1} = \frac{f_{1}}{s_{1}}, c_{2} = \frac{f_{2}}{s_{2}}, \ldots, c_{n} = \frac{f_{n}}{s_{n}}$. For each $i$ with $1 \le i \le n$, let $f_{i}s_{i}^{m_{i} - 1} = p_{i1}^{\delta_{i1}} \cdots p_{ir_{i}}^{\delta_{ir_{i}}}$ be the factorisation into primes and $R_{i}$ denotes the set containing the reduced irrational numbers reduced form the numbers $\sqrt[m_i]{p_{i1}^{\delta_{i1}}}, \ldots, \sqrt[m_i]{p_{ir_{i}}^{\delta_{ir_{i}}}}$. Let $R = \bigcup_{i = 1}^{n}R_{i}$. If $R$ is singleton set $R = S$. Otherwise, let $\sqrt[\beta_{1}]{p_{1}^{a_{1}}}, \sqrt[\beta_{2}]{p_{2}^{a_{2}}}, \ldots, \sqrt[\beta_{z}]{p_{z}^{a_{z}}}$ be the distinct elements of $R$ where $p_{i}$'s are primes. Clearly $gcd(\beta_{i}, a_{i}) = 1$ for all $i$ with $1 \le i \le z$. Now for each $i$ with $1 \le i \le z$, choose $u_{ij} \in \mathbb{N}$ such that $a_{i} = u_{i1} + u_{i2} + \cdots + u_{iv_{i}}$ and $\frac{\beta_{i}}{u_{ij}} = \theta_{j} \in \mathbb{N}$, $\forall j$ with $1 \le j \le v_{i}$. Further, for $1 \le i \le z$ let  $L_{i} = \{\sqrt[\theta_{j}]{p_{i}} \mid 1 \le j \le v_{i}\}$ and $L = \bigcup_{i = 1}^{z}L_{i}$. Now, if $L$ is singleton then set $L = S$. Otherwise, let $\sqrt[\mu_{1}]{q_{1}}, \sqrt[\mu_{2}]{q_{2}}, \ldots, \sqrt[\mu_{y}]{q_{y}}$ be the distinct elements of $L$ where $q_{i}$'s are primes for $1 \le i \le y$. For $1 \le k \le t$, let $q_{kl_{k}}$ be the distinct primes that appear $l_{k}$ times inside the radical signs, where $l_{1} + l_{2} + \cdots + l_{t} = y$. Now for each $q_{kl_{k}}$ with $1 \le k \le t$, let $\sqrt[\mu^{\prime}_{1}]{q_{kl_{k}}}, \sqrt[\mu^{\prime}_{2}]{q_{kl_{k}}}, \ldots, \sqrt[\mu_{l_{k}}^{\prime}]{q_{kl_{k}}}$ be the corresponding radicals and $\eta_{k} = lcm(\mu^{\prime}_{1}, \mu^{\prime}_{2}, \ldots, \mu^{\prime}_{l_{k}})$. Set $S = \{\sqrt[\eta_{k}]{q_{kl_{k}}} \mid 1 \le k \le t\}$. We claim that $S$ is a reduced set. To see this, observe that if the number 
$$\prod_{k = 1}^{t}(\sqrt[\eta_{k}]{q_{kl_{k}}})^{\epsilon_{k}} = \sqrt[(\eta_{1}\eta_{2} \cdots \eta_{t})]{\prod_{k = 1}^{t}q_{kl_{k}}^{(\epsilon_{k}\eta_{1} \cdots \eta_{k - 1}\eta_{k + 1} \cdots \eta_{t})}}$$ 
would be a rational number, where not all $\epsilon_{k}$'s are zero and $0 \le \epsilon_{k} \le \eta_{k} - 1$, then since $q_{kl_{k}}$'s are the distinct primes so it must be that $(\eta_{1}\eta_{2} \cdots \eta_{t}) \mid (\epsilon_{k}\eta_{1} \cdots \eta_{k - 1}\eta_{k + 1} \cdots \eta_{t})$ for each $k$ such that $\epsilon_{k} \ne 0$. But this means $\eta_{k} \mid \epsilon_{k}$ which is absurd. This completes the proof.
\end{proof}
  
Now we give an example to find the reduced set for the number $\sqrt[8]{12} - \frac{2}{3}\sqrt[27]{108} + \frac{1}{2}\sqrt{12}$ by applying the above lemma(3). Here, $c_{1} = 12$, $c_{2} = 108$, $c_{3} = 12$ and $m_{1} = 8$, $m_{2} = 27$, $m_{3} = 2$. Also, $f_{1} = 12$, $f_{2} = 108$, $f_{3} = 12$ and $s_{1} = 1$, $s_{2} = 1$, $s_{3} = 1$. Let $f_{1}s_{1}^{7} = f_{3}s_{3} = 12 = 2^{2}3$ and $f_{2}s_{2}^{26} = 108 = 2^{2}3^{3}$ be the factorisation into primes. Further, $R_{1} = \{\sqrt[4]{2}, \sqrt[8]{3}\}$, $R_{2} = \{\sqrt[27]{4}, \sqrt[9]{3}\}$, $R_{3} = \{\sqrt{3}\}$ and $R = \{\sqrt[4]{2}, \sqrt[8]{3}, \sqrt[27]{4}, \sqrt[9]{3}, \sqrt{3}\}$. Now observe that we have $\beta_{1} = 4$, $\beta_{2} = 8$, $\beta_{3} = 27$, $\beta_{4} = 9$, $\beta_{5} = 2$, $a_{1} = 1$, $a_{2} = 1$, $a_{3} = 2$, $a_{4} = 1$, $a_{5} = 1$ and $p_{1} = 2$, $p_{2} = 3$, $p_{3} = 2$, $p_{4} = 3$, $p_{5} = 3$. So we only have to look at the element $a_{3} = 2$ as rest all of $a_{i}$'s are 1. Since, $2 = 1 + 1$, so $L_{3} = \{\sqrt[27]{2}\}$ and $L_{1} = \{\sqrt[4]{2}\}$, $L_{2} = \{\sqrt[8]{3}\}$, $L_{4} = \{\sqrt[9]{3}\}$, $L_{5} = \{\sqrt{3}\}$. So we get $L = \{\sqrt[4]{2}, \sqrt[8]{3}, \sqrt[27]{2}, \sqrt[9]{3}, \sqrt{3}\}$. As, $lcm(4, 27) = 108$ and $lcm(8, 9, 2) = 72$, so we get that $S = \{\sqrt[108]{2}, \sqrt[72]{3}\}$ and $\sqrt[8]{12} - \frac{2}{3}\sqrt[27]{108} + \frac{1}{2}\sqrt{12} \in \mathbb{Q}(\sqrt[108]{2}, \sqrt[72]{3})$.

$$\textbf{Proof of the Theorem 1}$$

\begin{proof}
Let $\alpha \in \mathcal{S}$. Then the lemma(3) guarantees that there exists natural number $k$ and a reduced set $S = \{\sqrt[m_{i}]b_{i} \mid 1 \le i \le k\}$ such that we have, $\alpha = \gamma_{0} + \gamma_{1}\sqrt[m_{k}]{b_{k}^{i_{1}}} + \cdots + \gamma_{M}\sqrt[m_{k}]{b_{k}^{i_{M}}}$, where $\gamma_{j} \in K - \{0\}$ for $0 \le j \le M$, $0 \le i_{1} \le i_{2} \ldots \le i_{M - 1} < i_{M} < m_{k}$, $K = \mathbb{Q}(S - \{\sqrt[m_{k}]b_{k}\})$.

Now, suppose on the contrary we have $\alpha = \frac{p}{q}$ for some $p \in \mathbb{Z}$ and $q \in \mathbb{N}$. Consider the polynomials $f(X) = q\gamma_{M}X^{i_{M}} + \ldots + q\gamma_{1}X^{i_{1}} + q\gamma_{0} - p$ and $g(X) = X^{m_{k}} - b_{k}$. Since, $\sqrt[m_{k}]{b_{k}}$ is a common zero for both of the polynomials, so $gcd(f(X), g(X)) = h(X)$ exists in $K[X]$. Since $h(X)$ divides $g(X)$ thus, the zeros of $h(X)$ are of the form $\omega\sqrt[m_{k}]{b_{k}}$ where $\omega$ is some $m_{k}$-th root of unity. If $h_{0} \in K - \{0\}$ be the constant term of the polynomial $h(X)$, then we have $\prod{(\omega\sqrt[m_{k}]{b_{k}})} = h_{0}$, where the product is taken over all zeros of $h(X)$. Now taking modulus on both sides, we obtain that $|h_{0}| = \sqrt[m_{k}]{b_{k}^{deg(h)}}$, that is, $\sqrt[m_{k}]{b_{k}^{deg(h)}} \in K$. 

Let $\mathcal{B} \subseteq H$ be a basis for the $\mathbb{Q}$ vector space $K$, where $H$ denotes the set $$\{\prod_{i = 1}^{k - 1}(\sqrt[m_{i}]{b_{i}})^{\epsilon_{i}} \mid 0 \le \epsilon_{i} \le (m_{i} - 1)\}.$$ We claim that $Tr_{K/\mathbb{Q}}(\sqrt[m_{k}]{b_{k}^{deg(h)}}\beta)$ is nonzero for at least one $\beta$ in $\mathcal{B}$, where $Tr_{K/\mathbb{Q}} : K \rightarrow \mathbb{Q}$ is the well known trace function. It is because, if $Tr_{K/\mathbb{Q}}(\sqrt[m_{k}]{b_{k}^{deg(h)}}\beta) = 0$ $\forall \beta \in \mathcal{B}$, then since $\{\sqrt[m_{k}]{b_{k}^{deg(h)}}\beta \mid \beta \in \mathcal{B}\}$ forms a basis for the $\mathbb{Q}$ vector space $K$, so for any $u \in K$ with $u = \sum_{\beta \in \mathcal{B}}c_{\beta}(\sqrt[m_{k}]{b_{k}^{deg(h)}}\beta)$ for some $c_{\beta} \in \mathbb{Q}$, we get that $$Tr_{K/\mathbb{Q}}(u) = Tr_{K/\mathbb{Q}}(\sum_{\beta \in \mathcal{B}}c_{\beta}\sqrt[m_{k}]{b_{k}^{deg(h)}}\beta) = \sum_{\beta \in \mathcal{B}}c_{\beta}Tr_{K/\mathbb{Q}}(\sqrt[m_{k}]{b_{k}^{deg(h)}}\beta) = 0,$$ whereas $Tr_{K/\mathbb{Q}}(1) = |\mathcal{B}|$. Let $\beta^{*} \in \mathcal{B}$ be such that $Tr_{K/\mathbb{Q}}(\sqrt[m_{k}]{b_{k}^{deg(h)}}\beta^{*}) \ne 0$. Now if $m(X) = X^{d} + a_{d - 1}X^{d - 1} + \cdots + a_{1}X + a_{0}$ be the minimal polynomial for $\sqrt[m_{k}]{b_{k}^{deg(h)}}\beta^{*}$ over $\mathbb{Q}$, then certainly $a_{d - 1} = -\frac{d}{|\mathcal{B}|}Tr_{K/\mathbb{Q}}(\sqrt[m_{k}]{b_{k}^{deg(h)}}\beta^{*})$ is non-zero. But from the very next discussion on the lemma(2) we know that the minimal polynomial for $\sqrt[m_{k}]{b_{k}^{deg(h)}}\beta^{*}$ over $\mathbb{Q}$ is of the form $X^{s} - (\sqrt[m_{k}]{b_{k}^{deg(h)}}\beta^{*})^{s}$ for some $s \in \mathbb{N}$. Since the monic minimal polynomial is unique, therefore we must have $s = d$ and $a_{d - 1} = a_{0}$, that is $d = 1$. Hence we get that $(\sqrt[m_{k}]{b_{k}})^{deg(h)}\beta^{*} \in \mathbb{Q}$. Since $deg(h) < m_{k}$ and $\beta^{*} \in H$, so this contradicts the fact that $S$ is a reduced set, completing the proof of the theorem.
\end{proof}

\end{document}